\documentclass{amsart}
\usepackage{amssymb,amsmath,latexsym}
\usepackage{times}

\numberwithin{equation}{section}

\newtheorem{thm}{Theorem}[section]

\newtheorem{lem}[thm]{Lemma}

\newtheorem{cor}[thm]{Corollary}
\newtheorem{rem}[thm]{Remark}
\newtheorem{conj}[thm]{Conjecture}

\newcommand{\del}{\backslash}

\newcommand{\si}{\hbox{\rm si}}

\newcommand{\co}{\hbox{\rm co}}

\title[Excluded minors]{The excluded minors for the class of matroids that are graphic or bicircular lift }

\date{\today}

\author[Rong Chen]{Rong Chen}
\address{Center for Discrete Mathematics, Fuzhou University, Fuzhou, P. R. China.}

\thanks{This research was partially supported by NSFC (No. 11471076).}

\begin{document}
\begin{abstract}
Bicircular lift matroids  are a class of matroids defined on the edge set of a graph. For a given graph $G$, the circuits of its bicircular lift  matroid  $L(G)$ are the edge sets of those subgraphs of $G$ that contain at least two cycles, and are minimal with respect to this property. For each cycle $C$ of $G$, since $L(G)/C$ is graphic and  most graphic matroids are not bicircular lift, the class of bicircular lift  matroids  is not minor-closed. In this paper, we prove that the class of matroids that are graphic or bicircular lift  has a finite list of excluded minors.
\end{abstract}




\maketitle

\section{Introduction}
We assume that the reader is familiar with fundamental definitions in matroid and graph theory. All definitions in matroid theory that are used but not defined in the paper follow from Oxley's book \cite{Oxley}. For a graph $G$, a set $X\subseteq E(G)$ is a {\sl cycle} if $G|X$ is a connected 2-regular graph.  bicircular lift  matroids  are a class of matroids defined on the edge set of a graph. For a given graph $G$, the circuits of its bicircular lift  matroid  $L(G)$ are the edge sets of those subgraphs of $G$ that contain at least two cycles, and are minimal with respect to this property. That is, the circuits of $L(G)$ consists of the edge sets of two edge-disjoint cycles with at most one common vertex, or three internally disjoint paths between a pair of distinct vertices.  

Bicircular lift matroids  are a special class of lift matroids that arises from biased graphs, where biased graphs and its lift matroids were introduced by Zaslavsky in \cite{Zas89,Zas91}. Let $BL$ denote the class of bicircular lift  matroids . For each cycle $C$ of $G$, since $L(G)/C$ is graphic and most graphic matroids are not in $BL$ by the following Lemma \ref{graph}, this class $BL$ is not minor-closed. But the union of $BL$ and the class of graphic matroids is a minor-closed class. Let $\overline{BL}$ denote this class. Irene Pivotto  \cite{Irene} conjectured

\begin{conj}
The class $\overline{BL}$ has a finite list of excluded minors.
\end{conj}

In this paper, we prove that the conjecture is true. In fact, we prove a stronger result. 

\begin{thm}\label{main-thm}
Let $M$ be an excluded minor of $\overline{BL}$. Then  
\begin{itemize}
\item either $M$ is a direct sum of the uniform matroid  $U_{2,4}$ and a loop,  or 
\item $M$ is $3$-connected with $r(M)\leq11$ and with $|E(M)|\leq 224$.
\end{itemize}
\end{thm}

In the rest of paper, we always let $M$ be an excluded minor of $\overline{BL}$. The paper is organized as follows. Some related definitions and basic results are given in Section 2. In Section 3, we prove that when $M$ is not connected, $M$ is a direct sum of the uniform matroid $U_{2,4}$ and a loop. In Section 4, we prove that if $M$ is connected then $M$ is 3-connected. In Section 5, we prove that if $M$ is 3-connected then $r(M)\leq11$ and $|E(M)|\leq 224$.

Unfortunately, the number of matroids with rank at most 11 and size at most 224 is massive. There are too many matroids! The bound is outside what we are able to check with a computer.  The search space is just too large. 

\section{Preliminaries}
Let $G$ be a graph. Set $|G|:=|V(G)|$. For a vertex $v$ of $G$, let $st_G(v)$ denote the set of all edges adjacent with $v$.  An edge of $G$ is a {\sl link} if its end-vertices are distinct; otherwise it is a {\sl loop}. Let $\mathrm{loop}(G)$ be the set consisting of loops of $G$. We say that $G$ is {\sl 2-edge-connected} if each edge of $G$ is contained in some cycle.  A graph obtained from graph $G$ with some edges of $G$ replaced by internally disjoint paths is a {\sl subdivision} of $G$.

Let $e,f\in E(G)$. If $\{e,f\}$ is a cycle, then $e$ and $f$ are {\sl a parallel pair}. A {\sl parallel class} of $G$ is a maximal subset $P$ of $E(G)$ such that any two members of $P$ are a parallel pair and no member is a loop. Moreover, if $|P|\geq2$ then $P$ is {\sl non-trivial}; otherwise $P$ is {\sl trivial}. Let $\si(G)$ denote the graph obtained from $G$ by deleting all loops and all but one distinguished element of each non-trivial parallel class. Obviously, the graph we obtain is uniquely determined up to a renaming of the distinguished elements. If $G=\si(G)$, then $G$ is {\sl simple}. 

Two elements are a {\sl series pair} of a graph (or matroid $N$) if and only if  each cycle (or circuit) can not intersect them in exactly one element and they are contained in at least one cycle (or circuit).  A {\sl series class}  is a maximal set $X$ of  a graph (or matroid) such that every two elements of $X$ form a series pair. Let $\co(G)$ (or $\co(N)$) denote a graph (or matroid) obtained from $G$ (or $N$) by contracting all cut-edges (or coloops) from $G$ (or $N$) and then, for each series class $X$, contracting all but one distinguished element of $X$. Obviously, the graph we obtain is uniquely determined up to a renaming of the distinguished elements. We say that  $G$ is {\sl cosimple} if $G$ has no cut-edges or non-trivial series classes. 

\begin{lem}\label{series-pair}
Assume that $L(G)$ has at least two circuits. Then $\{e,f\}$ is a series pair of $L(G)$ if and only if $\{e,f\}$ is a series pair of $G$.
\end{lem}

\begin{proof}
First we prove the ``if" part. Since each cycle of $G$ can not contain exactly one edge of $\{e,f\}$,  each circuit of $L(G)$ can not contain exactly one element of  $\{e,f\}$. So $\{e,f\}$ is a series pair of $L(G)$.

Secondly we prove the ``only if" part. Assume otherwise. Then there are cycles $C_e,C_f$ of $G$ with $\{e\}=C_e\cap \{e,f\}, \{f\}=C_f\cap \{e,f\}$. On the other hand, since $L(G)$ has at least two circuits, some circuit in $L(G)$ does not contain $f$. So, besides $C_e$ there is another cycle $C$ of $G$ with $f\notin C$. Hence, there is a circuit $X$ of $L(G)$ with $e\in X\subseteq C_e\cup C$. Since $f\notin C_e\cup C$, we have that $\{e,f\}$ is not a series pair of $L(G)$, a contradiction.
\end{proof}

Note that when $L(G)$ has only one circuit, Lemma \ref{series-pair} is not true.

\begin{rem}\label{remark:series-pair}
Note that a matroid $N$ has at least two circuits if and only if $r^*(N)\geq2$. Hence, by Lemma \ref{series-pair}, when $r^*(L(G))\geq2$, the set $\{e,f\}$ is a series pair of $L(G)$ if and only if $\{e,f\}$ is a series pair of $G$.
\end{rem}

Let $G$ a connected graph with cycles. Since a connected spanning subgraph of $G$ with a unique cycle is a basis of $L(G)$, we have $r(L(G))=|G|$ and $st_G(v)-\mathrm{loop}(G)$ is a union of cocircuits of $L(G)$ for each vertex $v$ of $G$. Moreover, when $r^*(L(G))\geq2$, by Remark \ref{remark:series-pair} we have $r(\co(L(G)))=|\co(G)|$. In the rest of the paper, we will use these properties frequently without reference.

Given a set  $X$ of edges, let $G|X$ denote the subgraph of $G$ with edge set $X$ and no isolated vertices. Let $(X_1,X_2)$ be a partition of $E(G)$ with $V(G|X_1)\cap V(G|X_2)=\{u_1,u_2\}$. We say that $G'$ is obtained by a {\sl Whitney switching} on $G$ on $\{u_1,u_2\}$ if $G'$ is a graph obtained by identifying vertices $u_1,u_2$ of $G|X_1$ with vertices $u_2,u_1$ of $G|X_2$, respectively. A graph $G'$ is {\sl 2-isomorphic} to $G$ if $G'$ is obtained from $G$ by a sequence of the operations: Whitney switchings, identifying two vertices from distinct components of a graph, or partitioning a graph into components each of which is a block of the original graph. For graphic matroids, Whitney \cite{Whitney}  proved

\begin{thm}(Whitney's\ $2$-Isomorphism\ Theorem.)
Let $G_1$ and $G_2$ be graphs. Then $M(G_1)\cong M(G_2)$ if and only if $G_1$ and $G_2$ are $2$-isomorphic.
\end{thm}

Let $N$ be a bicircular lift  matroid . If $G$ is a graph satisfying $N=L(G)$, then we say that $G$ is a {\sl bicircular lift  graphic representation} of $N$.  Evidently, by Whitney's 2-Isomorphic Theorem, each graph that is 2-isomorphic to $G$ is a bicircular lift  graphic representation of $N$. So, we can assume that $G$ is connected. In fact,   we proved  

\begin{thm}(\cite{Chen15}, Corollary 1.3.)\label{non-eq-repre}
Let $G_1$ and $G_2$ be connected graphs with $L(G_1)=L(G_2)$ and such that $L(G_1)$ has at least two circuits. If $|\co(G_1)|\geq5$ then $G_1$ and $G_2$ are $2$-isomorphic.
\end{thm}

The following obvious results about bicircular lift  matroids  will be used without reference.  
\begin{itemize}
\item[(a)] $L(G)$ has no loops.  
\item[(b)] At most one component of $L(G)$ has circuits. 
\item[(c)] $L(G)$ is connected if and only if $G$ is 2-edge-connected and has at least two cycles.
\end{itemize}


\section{The Non-connected Case}
In this section, we prove that if an excluded minor $M$ is not connected then it is a direct sum of $U_{2,4}$ and a loop. To prove this, first we need to prove that matroids $M^*(K_5), M^*(K_{3,3}), F_7$ and $F_7^*$ are excluded minors of $\overline{BL}$.

Tutte \cite{Tutte} proved

\begin{thm}(\cite{Oxley}, Theorem 10.3.1.)\label{ex-mi-of-bi}
A matroid is graphic if and only if it has no minor isomorphic to $U_{2,4}, F_7, F^*_7, M^*(K_5)$ and $M^*(K_{3,3})$.
\end{thm}

Let $\mathcal{E}x$ be the set of excluded minors of $\overline{BL}$.

\begin{lem}\label{K_5^*}
$M^*(K_5)\in\mathcal{E}x$.
\end{lem}

\begin{proof}
Since $M^*(K_5)$ is an excluded minor of the class of graphic matroids by Theorem \ref{ex-mi-of-bi}, it suffices to show that $M^*(K_5)$ has no graphic bicircular lift  representation. Assume to the contrary that $M^*(K_5)=L(G)$ for some graph $G$. Evidently, $|G|=6$ and $G$ has at most one loop. Since $M^*(K_5)$ has no triangle, each non-trivial parallel class of $G$ has exactly two edges. For each 4-element circuit $C$ in $M^*(K_5)$, the graph $G|C$ has three possible structures: (1) two non-trivial parallel classes, (2) a union of a triangle and a loop, and (3) a theta-subgraph with exactly four edges. Note that each element is in exactly two 4-element circuits in $M^*(K_5)$ and the element is the unique common element of the two 4-element circuits. Since $G$ can not have a loop and a 2-element parallel class at the same time, $G$ has no loops. That is, (1) or (3) happens, so the matroid $M^*(K_5)$ has two 4-element circuits with at least two common elements, which is not possible.
\end{proof}

\begin{lem}\label{K_{3,3}^*}
$M^*(K_{3,3})\in\mathcal{E}x$.
\end{lem}

\begin{proof}
Since $M^*(K_{3,3})$ is an excluded minor of the class of graphic matroids by Theorem \ref{ex-mi-of-bi}, it suffices to show that $M^*(K_{3,3})$ has no graphic bicircular lift representation. Assume to the contrary that $M^*(K_{3,3})=L(G)$ for some graph $G$. Evidently, $|V(G)|=4$ and $G$ has at most one loop. For each 3-element circuit $C$ in $M^*(K_{3,3})$, either $G|C$ is a parallel class with exactly three edges or $G|C$ is a union of a 2-element parallel class and a loop. Moreover, since each element is in exactly two 3-element circuits, there are triangles $C_1, C_2$ of $M^*(K_{3,3})$ with exactly one common element such that $L(G)|C_1\cup C_2=U_{2,5}$, a contradiction as $M^*(K_{3,3})$ has no $U_{2,5}$-minors.
\end{proof}

Using a similar strategy as the proof of Lemma \ref{K_{3,3}^*} we can prove

\begin{lem}\label{F_7}
$F_7\in\mathcal{E}x$.
\end{lem}

\begin{lem}\label{F_7^*}
$F_7^*\in\mathcal{E}x$.
\end{lem}

\begin{proof}
Since $F_7^*$ is an excluded minor of the class of graphic matroids by Theorem \ref{ex-mi-of-bi}, it suffices to show that $F_7^*$ has no graphic bicircular lift representation. Assume to the contrary that $F_7^*=L(G)$ for some graph $G$. Evidently, $|V(G)|=4$ and $G$ has at most one loop. Since $F_7^*$ has no triangle, each non-trivial parallel class in $G$ has exactly two edges and if $G$ has a loop then $G$ has no non-trivial parallel class. Hence, if $G$ has a loop, then $G$ is a union of $K_4$ and a loop; so $L(G)$ has 5-element circuits, a contradiction as $F_7^*$ has no 5-element circuits. Hence, $G$ has no loops.  Moreover, since $F_7^*$ has no 5-element circuits and the simple graph $\si(G)$ of $G$ is connected, $\si(G)$ is a 4-element cycle or a union of a triangle $T$ and a cut-edge $e$, for otherwise some parallel class of $G$ has three elements. When $\si(G)$ is a 4-element cycle, since $G$ has at least one 2-element cycle, $F_7^*$ has a 5-element cycle, a contradiction. So the later case happens. Since $G$ is 2-edge-connected, there is an edge $f$ of $G$ such that $\{e,f\}$ is a cycle of $G$. Hence, $T\cup\{e,f\}$ is a 5-element circuit of $F_7^*$, which is not possible.
\end{proof}

By Theorem \ref{ex-mi-of-bi} and Lemmas \ref{K_5^*}-\ref{F_7^*}, we may assume that all excluded minors of $\overline{BL}$ and non-graphic bicircular lift  matroid  have a $U_{2,4}$-minor. The result will be used without reference.

A matroid is {\sl free} if it has no circuits. Let $N_1,N_2$ be matroids on disjoint sets. The {\sl direct sum} of $N_1,N_2$, denoted by $N_1\oplus N_2$, is defined on the ground set $E(N_1)\cup E(N_2)$ with $\mathcal{C}(N_1\oplus N_2)=\mathcal{C}(N_1)\cup\mathcal{C}(N_2).$ Let $G_1,G_2$ be vertex-disjoint graphs. The {\sl direct sum} of $G_1,G_2$, denoted by $G_1\oplus G_2$, is a graph with $V(G_1\oplus G_2)=V(G_1)\cup V(G_2)$ and $E(G_1\oplus G_2)=E(G_1)\cup E(G_2)$.

Recall that  $M$ is an excluded minor of $\overline{BL}$.

\begin{thm}\label{1-con}
Either $M$ is connected or $M$ is a direct sum of $U_{2,4}$ and a loop.
\end{thm}

\begin{proof}
Assume that $M=M_1\oplus M_2$ for some matroids $M_1,M_2$.  Assume that $M_1$ is free. Let $G_1$ be a tree with $E(G_1)=E(M_1)$, and let $G_2$ be a graphic representation or graphic bicircular lift representation of $M_2$. Then $G_1\oplus G_2$ is a graphic representation or graphic bicircular lift representation of $M$, a contradiction. So neither $M_1$ nor $M_2$ is free. 

Since $M$ has a $U_{2,4}$-minor,  one of $M_1$ and $M_2$ (say $M_1$) has a $U_{2,4}$-minor. Since $M_2$ has a circuit, $M$ contains a minor as a direct sum of $U_{2,4}$ and a loop. Moreover, since a direct sum of $U_{2,4}$ and a loop is in $\mathcal{E}x$, the matroid $M$ is a direct sum of $U_{2,4}$ and a loop. 
\end{proof}


(The second paragraph of the proof of Theorem \ref{1-con} was given by the referee, which is much simpler than the one that the author gave.) Therefore, in the rest of the paper, we assume that $M$ is connected.

\section{The $2$-connected Case}
In this section, we prove that if an excluded minor $M$ is connected then it is 3-connected.

\begin{lem}\label{series-class-1}
Let $\{e,f\}$ be a series pair of a matroid $N$. Then $N/e$ is in $\overline{BL}$ if and only if $N$ is in $\overline{BL}$.
\end{lem}

\begin{proof}
Evidently, it suffices to prove the ``only if" part. Let $G$ be a graphic or bicircular lift representation of $N/e$. Then the graph obtained from $G$ by replacing $f$ with a 2-edge-path labelled by $\{e,f\}$ is a graphic or graphic bicircular lift representation of $N$. So $N$ is in $\overline{BL}$.
\end{proof}

By Lemma \ref{series-class-1} we have

\begin{cor}\label{series-class}
$M$ has no series classes.
\end{cor}

\begin{lem}\label{K_4}
The graphic matroid of a subdivision of $K_4$ is not bicircular lift.
\end{lem}

\begin{proof}
By Lemma \ref{series-class-1} it suffices to show that $M(K_4)$ is not bicircular lift. Assume to the contrary that $M(K_4)=L(G)$ for some graph $G$. For each triangle $C$ in $K_4$, the graph $G|C$ is  a theta-graph or a union of a loop and a 2-element cycle. Moreover, since each edge of $K_4$ is in exactly two triangles and $G$ has at most one loop, there are triangles $C_1, C_2$ of $K_4$ with exactly one common element such that $L(G)|C_1\cup C_2=U_{2,5}$, a contradiction as $M(K_4)$ has no $U_{2,5}$-minors.
\end{proof}

For an integer $n\geq2$, let $K_2^n$ be the graph obtained from $K_2$ with its unique edge replaced by $n$ parallel edges.  A subdivision of $K_2^3$ is a {\sl theta-graph}.

\begin{lem}\label{graph}
If the graphic matroid of a $2$-edge-connected graph $H$ is bicircular lift, then $H$ is a subdivision of $K_2^n$ for some integer $n\geq2$.
\end{lem}

\begin{proof}
Assume $M(H)=L(G)$ for some graph $G$. Since bicircular lift  matroid  can not have loops, $H$ has no loops. Assume that there are cycles $C_1,C_2$ in $H$ with at most one common vertex. Then there is no circuit in $M(H)$ contained in $C_1\cup C_2$ and intersecting with $C_1$ and $C_2$.  On the other hand, since $G|C_i$ is a theta-graph or a union of two cycles with at most one common vertex for each $1\leq i\leq2$, some circuit $C$ in $L(G)$ intersects $C_1,C_2$ with $C\subseteq C_1\cup C_2$,  a contradiction. So $H$ hs no such cycles $C_1,C_2$, implying that $H$ is 2-connected as $H$ is 2-edge-connected.

Let $C$ be a cycle of $H$. Assume $H\neq C$. Since $H$ is 2-connected, there is a path $P$ such that $C\cup P$ is a theta-subgraph. When $H\neq C\cup P$, since no two cycles in $H$ have at most one common vertex, either $H$ is $K_2^n$-subdivision for some integer $n\geq4$ or $H$ contains a $K_4$-subdivision. By Lemma \ref{K_4} the graph $H$ is a subdivision of $K_2^n$.
\end{proof}

Let $G_1, G_2$ be vertex-disjoint graphs with $\{e\}=E(G_1)\cap E(G_2)$ and such that $e$ is a loop of $G_1, G_2$. Let $G_1 \oplus_1^e G_2$ be the graph obtained from $G_1$ and $G_2$ by identifying their end-vertices of $e$ and deleting $e$. We say that $G_1 \oplus_1^e G_2$ is a {\sl $1$-sum} of $G_1\del e$ and $G_2\del e$, which will be used in Section 5. Note that we do not define $G_1 \oplus_1^e G_2$ being a  $1$-sum of $G_1$ and $G_2$.

Let $N_1,N_2$ be matroids with $\{e\}=E(N_1)\cap E(N_2)$. Assume that $e$ is neither a loop nor a coloop of $N_1,N_2$. Let $N_1\oplus_2^e N_2$ or $N_1\oplus_2 N_2$ be the matroid with ground set $(E(N_1)\cup E(N_2))-e$ and
\[\begin{aligned}
\mathcal{C}(N_1\oplus_2^e N_2)=&\mathcal{C}(N_1\del e)\cup\mathcal{C}(N_2\del e)\\
&\cup\{(C_1-e)\cup(C_2-e)|\ e\in C_i\in\mathcal{C}(N_i)\ \text{for\ each}\ i\}.
\end{aligned}\]
We say that $N_1 \oplus_2^e N_2$ or $N_1\oplus_2 N_2$ is a {\sl $2$-sum} of $N_1$ and $N_2$.

Let $k$ be a positive integer. A partition $(X,Y)$ of the ground set of a matroid $N$ is a {\sl $k$-separation} if $|X|,|Y|\geq k$ and $r(X)+r(Y)-r(N)\leq k-1$.

\begin{lem}\label{2-con-}
Let $N, N_1,N_2$ be matroids with $N=N_1\oplus_2^e N_2$. If $N$ is bicircular lift, then there are graphs $G_1,G_2$ with $N=L(G_1\oplus_1^e G_2)$ and with $N_i=L(G_i)$ for each $1\leq i\leq2$.
\end{lem}

\begin{proof}
Since every bicircular lift   matroid  has at most one component having circuits, without loss of generality we can assume that $N$ is connected. So $N_1,N_2$  are connected. Let $G$ be a connected graph with $N=L(G)$ and with $V(G|E(N_1)-e)\cap V(G|E(N_2)-e)$ as small as possible. Set $c:=|V(G|E(N_1)-e)\cap V(G|E(N_2)-e)|$. Assume that $c=1$. Let $G^+$ be the graph obtained by adding a loop  labelled $e$ incident with the vertex shared by $E(N_1)-e$ and $E(N_2)-e$ in $G$. Let $G_i=G^+|E(N_i)$ for each integer $1\leq i\leq2$. We claim that $N_i=L(G_i)$. Let $C$ be a circuit of $N$ intersecting $E(N_1)-e$ and $E(N_2)-e$, and let $e_{3-i}\in C\cap(E(N_{3-i})-e)$. Since $N=L(G)$, we have that $C\cap E(N_{3-i})$ is a cycle of $G_{3-i}$; so $G_i$ is isomorphic to $(G/C\cap (E(N_{3-i})-e_{3-i}))|E(N_i)-e+e_{3-i}$. Moreover, since $N_i\cong (N/C\cap (E(N_{3-i})-e_{3-i}))|E(N_i)-e+e_{3-i}$, the claim holds. 

Now we prove that $c=1$. Assume otherwise. Then $c\geq2$ as $G$ is connected. For $1\leq i\leq 2$, set $n_i:=|G|E(N_i)-e|$ and let $c_i$ be the number of components in $G|E(N_i)-e$. Since $G$ is connected and has cycles, $r(N)=n_1+n_2-c$. Since $G$ is chosen with $c$ as small as possible, each component of $G|E(N_i)-e$ shares at least two vertices with $G|E(N_{3-i})-e$. So $c\geq 2max\{c_1,c_2\}$. On the other hand, since $N_i$ has circuits, $G|E(N_i)-e$ has cycles; so $r(E(N_i)-e)=n_i-c_i+1$. Then 
\[\begin{aligned}
&r(E(N_1)-e)+r(E(N_2)-e)-r(N)\\
=\ &n_1-c_1+1+n_2-c_2+1-(n_1+n_2-c)\\
=\ &c-c_1-c_2+2\geq2.
\end{aligned}\]
Hence, $(E(N_1)-e,E(N_2)-e)$ is not a 2-separation of $N$, a contradiction. So $c=1$. 
\end{proof}

We need three more results to prove the main result of the section. 

\begin{lem}\label{con-minor}(\cite{Oxley}, Proposition 4.3.7.)
Let $N_1$ be a connected minor of a connected matroid $N$ and $f\in E(N)-E(N_1)$. Then at least one of $N\del f$ or $N/f$ is connected having $N_1$ as a minor.
\end{lem}

Bixby \cite{Bixby} proved that

\begin{lem}\label{Bixby}(\cite{Oxley}, Proposition 12.3.7.)
Let $N$ be a connected matroid having a $U_{2,4}$-minor and $e\in E(N)$. Then $N$ has a $U_{2,4}$-minor minor using $e$.
\end{lem}

\begin{lem}\label{loop-loop}(\cite{Chen15}, Corollary 5.)
Let $G_1, G_2$ be graphs with $L(G_1)=L(G_2)$, and let $e$ be a loop of both $G_1$ and $G_2$. Then $G_1$ and $G_2$ are $2$-isomorphic. 
\end{lem}

\begin{thm}\label{2-con}
If $M$ is connected then $M$ is $3$-connected.
\end{thm}

\begin{proof}
Assume to the contrary that $M=M_1\oplus_2^e M_2$ for some connected matroid $M_1,M_2$ with at least three elements.  Evidently, either $M_1$ or $M_2$ is non-graphic for otherwise $M$ is graphic. By symmetry we may assume that $M_1$ is non-graphic. Then Lemmas \ref{ex-mi-of-bi}-\ref{F_7^*} imply that $M_1$ has a $U_{2,4}$-minor.  We claim that there is a graph $G_1$ with $M_1=L(G_1)$ such that $e$ is a loop of $G_1$.  By Lemma \ref{con-minor} for each $f\in E(M_2)-e$ either $M/f$ or $M\del f$ is connected having $U_{2,4}$ as a minor. Without loss of generality assume that the former case happens. Since $M/f=M_1\oplus_2^e (M_2/ f)$ and $M/f$ is bicircular lift, the claim follows from Lemma \ref{2-con-}.

We claim that $M_1=U_{2,4}$. Assume otherwise. Then by Lemmas \ref{con-minor} and \ref{Bixby} for some $f\in E(M_1)-e$ either $M\del f$ or $M/f$ is connected having $U_{2,4}$ as a minor. Without loss of generality assume that the latter case happens. Since $M/f$ is bicircular lift, by Lemma \ref{2-con-} there are graphs $G'_1,G_2$ with $M/f=L(G'_1\oplus_1^e G_2)$, $M_1/f=L(G'_1)$ and with $M_2=L(G_2)$. Since there is a graph $G_1$ with $M_1=L(G_1)$ such that $e$ is a loop of $G_1$, by Lemma \ref{loop-loop} we have that $G_1/f$ and $G'_1$ are 2-isomorphic, that is, a set is a cycle of $G_1/f$ if and only if it is a cycle of $G'_1$. Moreover, since $M/f=L(G'_1\oplus_1^e G_2)$ and $M_1=L(G_1)$, it is easy to verify that $M=L(G_1\oplus_1^e G_2)$, a contradiction. So $M_1=U_{2,4}$.

If $M_2$ is non-graphic, then by symmetry $M_2=U_{2,4}$, which is not possible as $M=U_{2,4}\oplus_2 U_{2,4}$ is  bicircular lift. So assume that $M_2$ is graphic. Since $U_{2,4}\oplus_2 U_{1,k}$ is bicircular lift for each integer $k\geq3$, we may further assume that $r(M_2)\geq2$. Hence, by the definition of 2-sum there is an element $f\in E(M_2)-e$ such that $M/f$ has a $U_{2,4}$-minor. Since $M/f$ is bicircular lift and has no coloops, $M/f$ is connected; so $M_2/f$ is connected. Moreover, by Lemma \ref{2-con-} the matroid $M_2/f$ is bicircular lift. Since $M_2$ is graphic, by Lemma \ref{graph}  there is a $K_2^n$-subdivision  $H$ for some integer $n\geq2$ with $M_2/f=M(H)$. Since $M$ has no series pairs by Corollary \ref{series-class},  each series pair of $M_2$ must contain $e$, so the graph $H$ has at most one 2-edge path and the path must contain $e$ when it exists. (Note that in the paper we do not see cycles as paths.) On the other hand, since $M/f=M_1\oplus_2^e M_2/f$ is bicircular lift, Lemma \ref{2-con-} implies that $e$ is a loop of some graphic bicircular lift representation of $M_2/f$. So $H=K_2^n$, that is, $M_2/f=U_{1,n}$. Hence, $\si(M_2)=U_{2,3}$ as $M_2$ is a connected graphic matroid. Since $U_{2,4}\oplus_2 U_{2,3}$ is bicircular lift, $M_2$ has at least one non-trivial parallel class. We claim that $M_2$ has a unique non-trivial parallel class and the non-trivial parallel class contains $e$. Assume otherwise. Then there is an element $e'\in E(M_2)-e$ such that $M/e'$ has a loop and $U_{2,4}$-minors, so $M/e'$ is not in $\overline{BL}$, which is not possible.  So the claim holds, implying that $M$ is bicircular lift, a contradiction.
\end{proof}

By Theorems \ref{1-con} and \ref{2-con}, in the rest of the paper we may assume that $M$ is 3-connected.

\section{The $3$-connected Case}
In this section, we prove that when an excluded minor $M$ is 3-connected and has a $U_{2,4}$-minor, we have $r(M)\leq 11$ and $|E(M)|\leq 224$. To prove this, we need introduce the following well-known result in matroid theory, which was proved by Seymour \cite{Seymour}.

\begin{thm}\label{splitter-th}(Splitter Theorem.) (\cite{Oxley}, Corollary 12.2.1.)
Let $N$ and $N'$ be $3$-connected matroids such that $N'$ is a minor of $N$ with at least four elements and if $N'$ is a wheel then $N$ has no larger wheel as a minor, while if $N'$ is a whirl then $N$ has no larger whirl as a minor. Then there is a sequence $N_0,N_1,\ldots, N_n$ of $3$-connected matroids such that $N_0\cong N$, $N_n\cong N'$, and for all integers $1\leq i\leq n$, $N_i$ is a single-element deletion or a single-element contraction of $N_{i-1}$.
\end{thm}

\begin{lem}\label{ele-del}
Let $N$ be a $3$-connected matroid with $U_{2,4}$-minors and $r^*(N)\geq3$. Then there is $e\in E(N)$ such that $\co(N\del e)$ is $3$-connected with $U_{2,4}$-minors.
\end{lem}

\begin{proof}
Evidently, $N$ is not a wheel. When $N$ is a whirl, the result is obviously true. So we may assume that $N$ is not a whirl. Let $N'$ be the largest whirl over all minors of $N$. By the Splitter Theorem there is a sequence $N_0,N_1,\ldots, N_n$ of $3$-connected matroids such that $N_0\cong N$, $N_n\cong N'$, and for all integers $1\leq i\leq n$, $N_i$ is a single-element deletion or a single-element contraction of $N_{i-1}$. Assume that some $N_i$ is a single-element deletion of $N_{i-1}$. Let $k$ be the minimum integer in $\{1,\ldots,n\}$ satisfying $N_k=N_{k-1}\del e_k$. Then $\co(N\del e_k)$ is 3-connected with $U_{2,4}$-minors. So we may assume that all $N_i$ are single-element contractions of $N_{i-1}$. Then $r^*(N')=r^*(N)\geq3$. Moreover,  since $r^*(U_{2,4})=2$, the matroid $N'$ is a whirl but $N'\ncong U_{2,4}$. Then there is an element $e$ in $N'$ such that $\co(N'\del e)$ is 3-connected with $U_{2,4}$-minors. Hence, $\co(N\del e)$ is 3-connected with $U_{2,4}$-minors.
\end{proof}

\begin{lem}\label{del-rank}
Let $N$ be a connected bicircular lift  matroid  with $r^*(N)\geq3$. Then  for each $e\in E(N)$ we have $r(\co(N\del e))\geq r(\co(N))-2$.
\end{lem}

\begin{proof}
Without loss of generality we may assume that $N$ has no series pairs. Let $G$ be a graph with $N=L(G)$. Then $N\del e=L(G\del e)$. Since $N$ has no series pairs, $N\del e$ has no coloops and by Remark \ref{remark:series-pair} the graph $G$ has no series pairs. Let $S$ be a series class in $N\del e$. Since $r^*(N\del e)\geq2$,  by Remark \ref{remark:series-pair} again $S$ is also a series class of $G\del e$. Hence, when $|S|\geq 3$, there is no way to put $e$ back in $G\del e$ such that $G$ has  no series pairs. So $|S|=2$. Moreover, when $N\del e$ has at least three series classes, it is also no way to put $e$ back to $G\del e$ such that $G$ has no series pairs. Hence, $N\del e$ has at most two series class and each series class has at most two elements. Hence, $r(\co(N\del e))\geq r(N)-2$ as $N\del e$ has no coloops.
\end{proof}


\begin{lem}\label{cor:del-rank}
When $r^*(M)\geq3$, there is an element $e\in E(M)$ with $r(\co(M\del e))\geq r(M)-3$ and such that $\co(M\del e)$ is $3$-connected with $U_{2,4}$-minors.
\end{lem}

\begin{proof}
Evidently, $M\neq U_{2,4}$ and when $M$ is a whirl the result holds. So we may assume that $M$ is not a whirl. By the Splitter Theorem there is an element $f\in E(M)$ such that $M/f$ or $M\del f$ is 3-connected with $U_{2,4}$-minors. If $M\del f$ is 3-connected with $U_{2,4}$-minors then the corollary holds. So we may assume that $M/f$ is 3-connected with $U_{2,4}$-minors. By Lemma \ref{ele-del} there is an element $e\in E(M/f)$ such that $\co(M/f\del e)$ is 3-connected with $U_{2,4}$-minors. Moreover, since $M/f$ is bicircular lift with $r^*(M/f)=r^*(M)\geq3$, by Lemma \ref{del-rank}  we have $r(\co(M/f\del e))\geq r(M/f)-2$. So $r(\co(M\del e))\geq r(M)-3$.
\end{proof}

\begin{lem}\label{contra-non-graphic}
Let $N$ be a matroid and $G$ a graph satisfying $N=L(G)$. Assume that $N$ is non-graphic with $r(\co(N)), r^*(N)\geq4$. Then for each element $e\in E(G)$, if $e$ is a link of $G$, then $N/e$ is non-graphic; and if $e$ is a loop or contained in some non-trivial parallel pair of $G$, then $N\del e$ is non-graphic.
\end{lem}

\begin{proof}
Evidently, we can assume that $N$ is connected. So $G$ is 2-edge-connected. When $e$ is a link of $G$ that is not contained in any non-trivial parallel pair, set $G_e:=G/e, N_e:=N/e$; when $e$ is a loop of $G$, set $G_e:=G\del e, N_e:=N\del e$; and when $e$ is contained in some non-trivial parallel pair, $G_e$ and $N_e$ can be defined as the first case or the second case. Evidently, $N_e=L(G_e)$. Next we prove that $N_e$ is non-graphic. Assume $N_e$ is graphic.

First consider the case that $G_e$ is 2-edge-connected. Then $N_e$ is a connected graphic bicircular lift  matroid . By Lemma \ref{graph} we have $r(\co(N_e))\in\{0,1\}$. On the other hand, since $r(\co(N)), r^*(N)\geq4$, by Lemma \ref{del-rank} we have $r(\co(N\del e))\geq 2$. Moreover, since $r(\co(N/e))\geq r(\co(N))-1\geq3$, we have $r(\co(N_e))\geq 2$, a contradiction.

Secondly, consider the case that $G_e$ is not 2-edge-connected. Since $G$ is 2-edge-connected, $G_e=G-e$ and the non-trivial parallel class $P$ containing $e$ has exactly two elements, say $e,f$, and $P$ is also a series class of $G$. So $G\del\{e,f\}$ is 2-edge-connected and
\[\begin{aligned}
\co(N\del e,f)=\co(N\del e/f)=\co(\co(N/f)\del e)\cong\co(\co(N)\del e).
\end{aligned}\]
Therefore, since $r^*(\co(N))=r^*(N)\geq4$ and $\co(N)$ is connected, Lemma \ref{del-rank} implies $r(\co(N\del e,f))\geq r(\co(N))-2\geq 2$. On the other hand, since $G\del\{e,f\}$ is 2-edge-connected and $N_e=N\del e$ is graphic, by Lemma \ref{graph} we have $r(\co(N\del e,f))\in\{0,1\}$, a contradiction.
\end{proof}

Let $n$ be a positive integer. Let $C$ be a cycle of a graph $G$ and $x_1,x_2,\ldots,x_n\in V(C)$. Assume that $(x_1,x_2,\ldots,x_n)$ occurs in this order circularly on $C$. For any two distinct $x_i$ and $x_j$, the cycle $C$ contains two $(x_i,x_j)$-paths. Let $C[x_i,x_{i+1},\ldots, x_j]$ denote the $(x_i,x_j)$-path in $C$ containing $x_i,x_{i+1},\ldots, x_j$ (and not containing $x_{j+1}$ if $i \neq j+1$), where subscripts are modulo $n$. Such path is uniquely determined when $n \geq 3$. Similarly, set
\[\begin{aligned}
C(x_i,x_{i+1},\ldots, x_j)&=C[x_i,x_{i+1},\ldots, x_j]- \{x_i,x_j\},\\
C(x_i,x_{i+1},\ldots, x_j]&=C[x_i,x_{i+1},\ldots, x_j]-\{x_i\}.
\end{aligned}\]

\begin{lem}\label{3-graphs}
When $r(M)\geq12$ and $r^*(M)\geq5$, there are connected graphs $G_1,G_2$ and elements $e_1,e_2\in E(M)$ such that the following hold.
\begin{itemize}
\item[(1)] $M\del e_1=L(G_1),  M\del e_2=L(G_2)$,
\item[(2)] $G_1\del e_2=G_2\del e_1$ and $G_1\del e_2$ is connected,
\item[(3)] $\si(G_1\del e_2)$ is $2$-connected or a 1-sum of a 2-connected graph and a link $e$ such that the parallel class of $G_1\del e_2$ containing $e$ has exactly two edges.
\end{itemize}
\end{lem}

\begin{proof}
Since $M$ is 3-connected and has $U_{2,4}$-minors, by Lemmas \ref{ele-del}-\ref{cor:del-rank} there are elements $e_1,e_2\in E(M)$ such that $\co(M\del e_1,e_2)$ is 3-connected and has a $U_{2,4}$-minor and with $r(\co(M\del e_1,e_2))\geq7$ and such that $e_1,e_2$ are not contained in a triad of $M$. Since $\co(M\del e_1,e_2)$ is non-graphic, by Theorem \ref{non-eq-repre}, there are connected graphs $G_1$, $G_2$ and 2-isomorphic connected graphs $G_{12}, G'_{12}$ such that
\[\begin{aligned}
&M\del e_1,e_2=L(G_{12})=L(G'_{12}),\\
&M\del e_1=L(G_1),\ \ M\del e_2=L(G_2),\\
&G_{12}=G_1\del e_2,\  \ G'_{12}=G_2\del e_1.
\end{aligned}\]
Since $\co(M\del e_1,e_2)$ is 3-connected and $e_1,e_2$ are not contained in a triad of $M$, {\bf (a)} $G_{12}$ is 2-edge-connected with  at most one loop; {\bf (b)} $\si(G_{12})$ is 2-connected or a 1-sum of a 2-connected graph and a link $e$ such that the parallel class $P_e$ of $G_{12}$ containing $e$ has exactly two edges, for otherwise $\co(M\del e_1,e_2)$ has a 2-separation. Evidently, (a)-(b) also hold for $G'_{12}$, and to prove the lemma it suffices to show that $G_{12}=G'_{12}$. Assume to the contrary that $G_{12}\neq G'_{12}$ and $G_{12}, G'_{12}$ are chosen such that one can obtained from the other by as few Whitney Switching as possible.

Since $r^*(M\del e_1,e_2)=r^*(M)-2\geq3$, we have $|G_{12}|=r(M\del e_1,e_2)\geq7$. Evidently, for $\{i,j\}=\{1,2\}$, the edge $e_i$ is neither a loop nor a cut-edge of $G_j$ for otherwise it is easy to prove $G_{12}=G'_{12}$. Let $X\subseteq E(M)-\{e_1,e_2\}$ such that {\bf (c)} $G_2|X\cup e_1$ is a theta-subgraph and $G_2|X$ is a cycle or a 1-sum of $e$ and a cycle. By (b) such $X$ obviously exists. No matter which case happens, let $C$ be the unique cycle in $G_2|X$. Since $G_{12}, G'_{12}$ are 2-isomorphic, $C$ is also a cycle of $G_1$.  Let $X_1, X_2\subseteq X$ such that $X_1\cup e_1$ and $X_2\cup e_1$ are cycles in $G_2$. Evidently, when $e\notin X$, we have $C=X$ and $(X_1,X_2)$ is a partition of $C$; and when $e\in X$, we have $e\in X_1\cap X_2$,  $C=X-e$ and $(X_1-e,X_2-e)$ is a partition of $C$. \\

\noindent{\bf Claim 1.}  {\sl Neither $G_1|X_1$ nor $G_1|X_2$ is a path.  }

\begin{proof}[Subproof.]
Assume to the contrary that $G_1|X_2$ is a path. Then $G_1|X_1$ and $G_1|X_2$ are paths having the same end-vertices as $C$ is a cycle of $G_{12}$ and $G'_{12}$. Let $x_1,y_1$ (and $x'_1,y'_1$) be the end-vertices of $G_1|X_1$ (and $G_2|X_1$). That is, the end-vertices of $e_1$ in $G_2$ are $x'_1, y'_1$. Since $G_{12}, G'_{12}$ are 2-isomorphic and $X_1$ and $X_2$ are paths in $G_{12}$ and $G'_{12}$, we have that $P$ is a path in $G_{12}$ joining $x_1,y_1$ if and only if it is a path in $G'_{12}$ joining $x'_1,y'_1$. (Note that the orders of edges in $G_{12}|P$ and $G'_{12}|P$ may be not the same.) Hence, since $M\del e_1,e_2=L(G_{12})=L(G'_{12})$, it is easily to check that the graph obtained from $G_{12}$ by adding the edge $e_1$ connecting $x_1$ and $y_1$ is a graphic bicircular lift representation of $M\del e_2$. Hence, by the choice of $G_{12}, G'_{12}$ we have $G_{12}=G'_{12}$, a contradiction.
\end{proof}

Claim 1 implies that $|X|\geq4$ when $e\notin X$. Moreover, by the choice of $G_{12}$ and $G'_{12}$ we have \\

\noindent{\bf Claim 2.}  {\sl For each graph $G'_1$ $2$-isomorphic to $G_1$, neither $G'_1|X_1$ nor $G'_1|X_2$ is a path.}\\

To prove Claim 3 we need two more definitions. A path is a {\sl $X_1$-path} if its edges are in $X_1$.  Let $f=uv$ be a link of a graph $H$, and let $w$ be the vertex obtained by contracting $f$. If for every 2-vertex-cut of $H/f$ containing $w$, say $\{w,z\}$,  either $\{u,z\}$ or $\{v,z\}$ is a 2-vertex-cut of $H$, then we say that {\sl no new 2-vertex-cuts appear when $f$ is contracted in $H$}. Evidently, for some link $f$ in $E(G_1)-e_2$ if neither $G_1/f|(X_1-f)$ nor $G_1/f|(X_2-f)$ is a path and no new 2-vertex-cuts appear when $f$ is contracted in $G_1$, then  for each graph $G_f^2$ that is $2$-isomorphic to  $G_1/f$, neither $G_f^2|X_1-f$ nor $G_f^2|X_2-f$ is a path. \\

\noindent{\bf Claim 3.}  {\sl There is an edge $f$ in $E(G_{12})$ such that at least one of the following holds.
\begin{itemize}
    \item[(I)] When $G_1$ is not simple, we have
     \begin{itemize}
    \item[(i)] $M\del e_1,f$ is non-graphic; and
    \item[(ii)]  for each graph $G_f^2$ that is $2$-isomorphic to  $G_1\del f$, neither $G_f^2|X_1-f$ nor $G_f^2|X_2-f$ is a path.
    \end{itemize}
   \item[(II)] When $G_1$ is simple, we have
   \begin{itemize}
    \item[(i)] $M\del e_1/f$ is non-graphic; and
    \item[(ii)]  for each graph $G_f^2$ that is $2$-isomorphic to  $G_1/f$, neither $G_f^2|X_1-f$ nor $G_f^2|X_2-f$ is a path.
    \end{itemize}
\end{itemize}}

\begin{proof}[Subproof.]
When $G_1$ is not simple, let $f$ be a loop or an edge in a non-trivial parallel pair. Since $M\del e_1$ is non-graphic and $r^*(M\del e_1)=r^*(M)-1\geq4$, by Lemma \ref{contra-non-graphic} the matroid $M\del e_1,f$ is non-graphic. Moreover, since $G_1\del f$ and $G_1$ have the same 2-vertex-cuts, (I) (ii) holds from Claim 2.  So (I) holds.

We may therefore assume that $G_1$ is a simple graph. (b) implies that $G_{12}$ is 2-connected, so by (c) we have $C=X$. For each edge $f\in E(G_{12})$,  Lemma \ref{contra-non-graphic} implies that $M\del e_1/f$ is non-graphic. So it suffices to show that (II) (ii) holds.

We claim that every edge in $G_1$ has at least one endpoint on $C$. Assume to the contrary that there is an edge $f$ of $G_1$ that has no endpoints on $C$.  Evidently, $f\neq e_2$ and $C$ is also a cycle of $G_1/f$ and neither $G_1/f|(X_1-f)$ nor $G_1/f|(X_2-f)$ is a path by Claim 2. Since $f$ and $C$ are vertex-disjoint in $G_1$, for each Whitney Switching in $G_1/f$ changing the order of edges in $C$ there is a corresponding Whitney Switching in $G_1$ playing the same role on $C$, then (II) (ii) follows from Claim 2.


First we consider the case $V(G_1)\neq V(G_1|C)$. Let $w\in V(G_1)-V(G_1|C)$. When $|st_{G_1}(w)|=2$, for each edge $f\in st_{G_1}(w)-e_2$, no new 2-vertex-cuts appears when $f$ is contracted in $G_1$. So we may assume that $|st_{G_1}(w)|\geq3$. Then there is a unique minimal path $P_w$ with $E(P_w)\subsetneq C$ such that all vertices incident with $w$ are in $V(P_w)$. We say that $P_w$ is the {\sl neighbour path} of $w$ in $C$. Let $h$ be a vertex in $V(G_1)-(w\cup V(G_1|C))$ or an edge in $E(G_1)-C$ with its end-vertices on $C$. We say that the neighbourhoods of $w,h$ are {\sl crossing} on $C$ if there are distinct vertices $u,v,a,b\in V(G_1|C)$ with $u,v$ adjacent with $w$ and $a,b$ adjacent with or incident to $h$ such that $(u, a, v, b)$ appears in this order circularly on $C$.

We claim that if the neighbourhoods of $w,h$ are crossing on $C$ then (II) (ii) holds. When $u$ is the unique vertex in $C(a,u,b)$ adjacent with $w$, it is obvious that there is no 3-vertex-cut $\{w,u,x\}$ of $G_1$ with $x$ is on the $(a,b)$-path of $C$ containing $u$; moreover, since the neighbourhoods of $w,h$ are crossing on $C$, vertices $a$ and $b$ are connected in $G_1-\{w,u,y\}$ for each internal vertex $y$ on the $(a,b)$-path of $C$ not containing $u$. Hence, no new 2-vertex-cuts that change the order of edges on $C$ appear after the edge $wu$ is contracted. So by symmetry we may assume that there are vertices $u_1\in C(a,u,b)-u$ and $v_1\in C(a,v,b)-v$ adjacent with $w$ such that $(u, u_1, a, v_1, v, b)$ appears in the order circularly in $C$. Without loss of generality we may further assume that $e_2\notin\{wu,wv_1\}$ and $u$ and $v_1$ are the unique vertices in $C(b,u,u_1)$ and $C(a,v_1,v)$ respectively, which are adjacent with $w$. Hence, {\bf (d)} when $wu$ (or $wv_1$) is contracted, besides the vertex obtained by contracting the edge, the other vertex contained in a new 2-vertex-cut is in $C(u_1,a]$ (or $C(v,b]$). Assume that (II) (ii) does not hold when $f\in\{wu,wv_1\}$. Since no 2-vertex-cut of $G_1/wu$ or $G_1/wv_1$ has one vertex in $C(b,u,u_1,a)$ and the other in $C(a,v_1,v,b)$, by Claim 2 and (d) the paths $C[b,u,u_1,a]$ and $C[a,v_1,v,b]$ intersect $X_1$ and $X_2$. Moreover, since (II) (ii) does not hold when $f\in\{wu,wv_1\}$,  by (d) we have that {\bf (e)} for some graph $G_u$ 2-isomorphic to $G_1/wu$, the path $G_u|C[a,v_1,v,b]$ is a union of a $X_1$-path and a $X_2$-path and the two edges in $G_u|C$ incident to $b$ are in the same $X_i$; and {\bf (f)} for some graph $G_{v_1}$ 2-isomorphic to $G_1/wv_1$, the path $G_{v_1}|C[b,u,u_1,a]$ is a union of a $X_1$-path and a $X_2$-path and the two edges in $G_{v_1}|C$ incident to $a$ are in the same $X_j$. On the other hand, since the two vertices in a 2-vertex-cut of $G_1, G_1/wu$ or $G_1/wv_1$ are in $C[b,u]$, $C(u,u_1,a)$, $C[a,v_1,v]$ or $C[v,b]$, combined with (e) and (f), there is a graph $G'_1$ 2-isomorphic to $G_1$ such that $G'_1|X_1$ and $G'_1|X_2$ are paths, a contradiction to Claim 2.

We may therefore assume that the neighbourhoods of $w,h$ are non-crossing on $C$. Hence, {\bf (g)} the end-vertices of a neighbor path in $C$ of a vertex in $V(G_1)-V(G_1|C)$ consist of a 2-vertex-cut of $G_1$, and $|C|\geq5$ for otherwise $|V(G_1)|\leq6$ as $|st_{G_1}(w)|\geq3$ for each $w\in V(G_1)-V(G_1|C)$.

Assume that (II) (ii) does not hold for each edge in $E(G_{12})$. Let $G'_1$ be a graph 2-isomorphic to $G_1$ with the number of vertex disjoint $X_1$-paths as small as possible. Since the neighborhoods of $w, h$ are non-crossing on $C$ and the degree of each vertex not in $C$ is at least three in $G_1$, no new 2-vertex-cut appears when an edge in some neighbor path is contracted. Hence, {\bf (h)} $G'_1|C$ is a union of exactly two vertex-disjoint $X_1$-paths and exactly two vertex-disjoint $X_2$-paths and for each edge $p$ in some neighbor path of a vertex in $V(G'_1)-V(G'_1|C)$ when $p\in X_i$ the two edges in $G'_1|C$ adjacent with $p$ are in $X_j$, where $\{i,j\}=\{1, 2\}$. So $1\leq|V(G'_1)-V(G'_1|C)|\leq2$. Assume that for some $w\in V(G'_1)-V(G'_1|C)$ we have $|st_{G'_1}(w)|=3$. Let $u,v$ be the end-vertices of the neighbor path $P_w$ of $w$ in $C$. Then by (g) and (h) by a Whitney Switching on the 2-vertex-cut $\{u,v\}$ in $G'_1$ the sets $X_1$ and $X_2$ become paths, a contradiction to Claim 2. Hence, for each $w\in V(G'_1)-V(G'_1|C)$ we have $|st_{G'_1}(w)|\geq4$. Using (h) again we have  $\{w\}=V(G'_1)-V(G'_1|C)$ and $G'_1|(C-P_w)$ is a $X_i$-path and $G'_1|P_w$ is a 3-edge path such that the internal edge is contained in $X_i$ and the other two are contained in $X_j$, where $\{i,j\}=\{1, 2\}$. On the other hand, since $|C|=|G_{12}|-1\geq6$, we have $|E(G'_1|C-P_w))|\geq3$. Let $f\in E(G'_1|(C-P_w))$. Since the neighborhoods of $w$ and each edge in $E(G'_1)-st_{G'_1}(w)-E(C)$ are non-crossing on $C$, for each graph $G_1^f$ that is 2-isomorphic to $G'_1/f$ the graph $G_1^f|P_w$ is a 3-edge path with its internal edge in $X_i$. Hence, (II) (ii) holds as $E(G'_1|(C-P_w))\subseteq X_i$.

Secondly we consider the case $V(G_1)=V(G_1|C)$. Then $|C|=|G_{12}|\geq7$.  So by Claim 1 there is an edge $f=uv\in C$ such that neither $G_1/f|(X_1-f)$ nor $G_1/f|(X_2-f)$ is a path. If no new 2-vertex-cuts appear after $f$ contracted, then  (II) (ii) holds. So we may assume that some new 2-vertex-cut appears when $f$ is contracted. Then there are edges $g=uu_1, h=vv_1$ in $E(G_{12})-C$ such that $u_1,v_1$ are not adjacent in $C$ and $(u, v_1,u_1,v)$ occurs in this order circularly in $C$ and such that the vertex obtained by contracting $f$ and some vertex in $C(v_1,u_1)$ consist of a 2-vertex-cut of $G_1/f$. Hence, by symmetry we may assume that the end-vertices of $e_1$ are in $C[u,v_1,u_1)$. Then $C[u,v_1,u_1]\cup\{g,e_1\}$ is a theta-subgraph. Moreover, since $C[u,v_1,u_1]\cup g$ is a cycle of $G_{12}$ and $V(G_1)\neq V(G_1|C[u,v_1,u_1])$, the result (II) (ii) follows from the first case.
\end{proof}

Let $f$ be an edge satisfying Claim 3. When Claim 3 (I) holds, set $M_f:=M\del f, G_f^1:=G_1\del f$; and when Claim 3 (II) holds, set $M_f:=M/f$,  $G_f^1:=G_1/f$. Let $G_f^2$ be a graph 2-isomorphic to $G_f^1$ that can be extended to a graph $G_f$ with $M_f=L(G_f)$. Since $M$ is an excluded minor of $\overline{BL}$, by Theorem \ref{non-eq-repre} such $G_f^2$ exists.  For each cycle $C'$ of $G_f^2\del e_2$, since $C'$ or $C'\cup f$ is a cycle of $G_{12}$ and $X_i\cup e_1$ is a cycle of $G_2$ for each $1\leq i\leq 2$, {\bf (i)} the set $(C'\cup X_i\cup e_1)-f$ is cyclic in $M_f$, where a subset $Y$ of the ground set of a matroid $N$ is {\sl cyclic} in $N$ if $N|Y$ has no coloops.

Since $M_f\del e_1,e_2=L(G_f^2\del e_2)$ and $M_f\del e_1,e_2$ has at least two circuits by the fact that  $r^*(M_f\del e_1,e_2)\geq r^*(M)-3\geq2$, there is a subgraph $P$ of $G_f^2\del e_2$ such that either $C\cup P$ is a theta-subgraph or $C$ and $P$ are cycles with at most one common vertex. When the latter case happens, since neither $G_f^2|X_1-f$ nor $G_f^2|X_2-f$ is a path by Claim 3, $(P\cup X_1\cup e_1)-f$ is not cyclic in $M_f$, a contradiction to (i). So $C\cup P$ is a theta-subgraph. Let $C_1,C_2$ be the cycles in $C\cup P$ containing $P$, and $P_i=E(C_i)-E(P)$ for $1\leq i\leq2$. Since $C\cup P$ is a theta-subgraph and for any $1\leq i,j\leq 2$ the set $(C_j\cup X_i\cup e_1)-f$ is cyclic in $M_f$ by (i), both $P_1-f$ and $P_2-f$ intersect $X_1$ and $X_2$; so $G_f|X\cup P\cup e_1$ is a $K_4$-subdivision and $C$ is a union of exactly two vertex-disjoint $X_1$-paths and exactly two vertex-disjoint $X_2$-paths. Moreover, since $M_f\del e_1,e_2=L(G_f^2\del e_2)$ and $M_f\del e_1,e_2$ has at least two circuits, $E(G_f^2)\neq X\cup E(P)$. Hence, there is another cycle $C'$ of $G_f^2\del e_2$ such that $(C'\cup X_i\cup e_1)-f$ is not cyclic in $M_f$, a contradiction to (i).
\end{proof}

Recall that  $\mathrm{loop}(G)$ is the set consisting of loops of a graph $G$.

\begin{lem}\label{lift-bicircular}
Let $N$ be a $3$-connected matroid with $r(N)\geq10$ and $r^*(N)\geq4$. Let $G$ be a connected graph with $E(N)=E(G)$. Assume that all proper minors of $N$ are in $\overline{BL}$ and there are $e_1,e_2\in E(N)$ such that the following statements hold.
\begin{itemize}
\item[(1)] $N\del e_1=L(G\del e_1),  N\del e_2=L(G\del e_2)$,
\item[(2)] $r(\co(N\del e_1)), r(\co(N\del e_2))\geq7$,
\item[(3)] $G\del e_1,e_2$ is $2$-edge-connected,
\item[(4)] $\si(G\del e_1,e_2)$ is $2$-connected or a $1$-sum of $K_2$ and a $2$-connected graph.
\end{itemize}
Then $N=L(G)$.
\end{lem}

\begin{proof}
Since $N\del e_2$ is connected and $N\del e_2=L(G\del e_2)$, the graph $G\del e_2$ is 2-edge-connected. Then $e_1$ is not a cut-edge of $G$. So by (1) we have \[|V(G)|=|V(G\del e_1)|=r(N\del e_1)=r(N)\geq10.\]

\noindent{\bf Claim 1.}  {\sl For each vertex $v$ of $G$, the set $st_G(v)-\mathrm{loop}(G)$ is a union of cocircuits of $N$. }

\begin{proof}[Subproof.]
Since $st_G(v)-\mathrm{loop}(G)-e_i$ is a union of cocircuits of $N\del e_i$ for each $1\leq i\leq 2$ by (1), the set $st_G(v)-\mathrm{loop}(G)$ is a union of cocircuits of $N$.
\end{proof}

\noindent{\bf Claim 2.}  {\sl Let $C_1,C_2$ be cycles of $G$ with at most one common vertex and with $e_1\in C_1, e_2\in C_2$. Then $C_1\cup C_2\in\mathcal{C}(N).$}

\begin{proof}[Subproof.]
Assume that there is another cycle $C_3$ of $G$ such that $C_2$ and $C_3$ have at most one common vertex.  Without loss of generality we may further assume that either $C_1\cup C_3$ is a theta-subgraph with $e_1\notin C_3$ or $C_1$ and $C_3$ have at most one common vertex. Since $C_2\cup C_3$ and $C_1\cup C_3$ are circuits of $N$ by (1), for any $f\in E(C_3)-E(C_1\cup C_2)$ the set $(C_1\cup C_2\cup C_3)-f$ contains a circuit of $N$. Since a circuit and cocircuit of a matroid can not have exactly one common element, by Claim 1 we have $C_1\cup C_2\in\mathcal{C}(N).$ So by symmetry we may assume that except $C_1,C_2$ each cycle of $G$ intersects $E(C_1)$ and $E(C_2)$, implying that $G$ is a simple graph. Hence, by (3) and (4) the graph $\si(G\del e_1,e_2)$ is $2$-connected. Moreover, since $N$ is 3-connected, by Claim 1 we have $\delta(G)\geq3$, where $\delta(G)$ is the minimum degree of $G$.

Assume that $C_1,C_2$ have a common vertex $v$. Let $P$ be a shortest path of $G$ with $v\notin P$ joining $C_1$ and $C_2$. Since $\si(G\del e_1,e_2)$ is $2$-connected, such $P$ exists. Let $C_3$ be the cycle of $G$ with $P\subseteq C_3\subseteq (C_1\cup C_2\cup P)-\{e_1,e_2\}$. Since $C_2\cup C_3$ and $C_1\cup C_3$ are circuits of $N$ by (1), for any $f\in E(P)-E(C_1\cup C_2)$ the set $(C_1\cup C_2\cup P)-f$ contains a circuit of $N$; so $C_1\cup C_2\in\mathcal{C}(N).$ So we may assume that $V(C_1)\cap V(C_2)=\emptyset$. Since $\delta(G)\geq3$ and except $C_1,C_2$ each cycle of $G$ intersects $E(C_1)$ and $E(C_2)$, {\bf (a)} the graph $G$ is a 3-connected 3-regular graph with $V(G)=V(C_1)\cup V(C_2)$, $|V(C_1)|=|V(C_2)|$ and such that $E(G)-E(C_1\cup C_2)$ is a perfect matching of $G$.

Assume that the claim is not true. Then for each $f\in E(G)-E(C_1\cup C_2)$ we have $C_1\cup C_2\cup f\in\mathcal{C}(N)$. Since $|C_1|=\frac{|V(G)|}{2}\geq5$, there is $f'\in E(G)-E(C_1\cup C_2)$ such that no three-edge path containing $e_1, f'$. Hence, $\co(G\del e_1,f')$ is 3-connected by (a), so for every graph $H$ 2-isomorphic to $G\del e_1,f'$ we have $\co(H)\cong\co(G\del e_1,f')$ and each series class of $G\del e_1,f'$ is a path of $G\del e_1,f'$ and $H$. On the other hand, since  by (2) and Lemma \ref{del-rank} we have $r(\co(N\del e_1,f'))\geq 5$, by Lemma \ref{graph} the matroid $N\del f'$ is non-graphic. So $N\del f'$ is bicircular lift. Moreover, since $N\del e_1,f'=L(G\del e_1,f')$, some graph $H$ 2-isomorphic to $G\del e_1,f'$ can be extended to a graphic bicircular lift representation of $N\del f'$ by Theorem \ref{non-eq-repre}. Since $\co(H)\cong\co(G\del e_1,f')$ and each series class of $G\del e_1,f'$ is a path of $G\del e_1,f'$ and $H$, it is no way to add $e_1$ to $H$ such that $C_1\cup C_2\cup f$ is a theta-subgraph or a handcuff for an edge $f\in E(G)-E(C_1\cup C_2\cup f')$. So $C_1\cup C_2\cup f\notin\mathcal{C}(N)$,a contradiction.
\end{proof}

\noindent{\bf Claim 3.}  {\sl Each cycle of $G$ containing $e_1,e_2$ is independent in $N$.}

\begin{proof}[Subproof.]
Assume to the contrary that some cycle $C$ of $G$ containing $e_1,e_2$ is dependent in $N$. Then $C\in\mathcal{C}(N)$ by Claim 1. Evidently, $N/e_1\del e_2=L(G/e_1\del e_2)$ and $r(\co(N/e_1\del e_2))\geq 6$ by (1) and (2). Since $r(\co(N/e_1))\geq r(\co(N/e_1\del e_2))\geq 6$, by Lemma \ref{graph} the matroid $N/e_1$ is non-graphic. So $N/e_1$ is  bicircular lift. Moreover, since $r^*(N)\geq4$ implies $r^*(N/e_1\del e_2)\geq3$, by (3) and Theorem \ref{non-eq-repre} some graph $H$ that is 2-isomorphic to $G/e_1\del e_2$ can be extended to a graphic bicircular lift representation of $N/e_1$, which is not possible, since $C-\{e_1,e_2\}$ is a forest in $H$, it is no way adding $e_2$ to $H$ such that $C-e_1$ is a theta-subgraph or a handcuff.
\end{proof}

Next, we prove that $N=L(G)$. To Prove the result, by Claims 2 and 3 it suffices to show that each theta-subgraph $T$ of $G$ containing $e_1,e_2$ is in $\mathcal{C}(N)$. Note that $e_1,e_2$ maybe a series pair  of $G|T$. Let $P$ be a path internally disjoint with $T$ with its end-vertices on $T$. By (4) we can further assume that $P$ is chosen such that there are cycles $C_1,C_2$ of $G$ with $e_1\in C_1\subseteq (T\cup P)-e_2$ and $e_2\in C_2\subseteq (T\cup P)-e_1$. Let $C$ be the cycle of $G$ with $C\subseteq (T\cup P)-\{e_1,e_2\}$. Evidently, such $C$ exists and $P\subseteq (C\cup C_1)\cap (C\cup C_2)$. Moreover, since $C\cup C_1, C\cup C_2\in\mathcal{C}(N)$, for each $f\in P$ there is a set $C'\in\mathcal{C}(N)$ with $C'\subseteq (T\cup P)-e$. Claim 1 implies that $C'=T$. Hence, the lemma holds.
\end{proof}

The following result follows immediately from Lemmas \ref{del-rank}, \ref{3-graphs} and \ref{lift-bicircular}.

\begin{lem}\label{rank-corank}
Either $r(M)\leq 11$ or $r(M^*)\leq4$.
\end{lem}


\begin{lem}\label{corank<4}
When $r(M^*)\leq4$, we have $r(M)\leq 9$.
\end{lem}

\begin{proof}
Evidently, $r(M^*)\geq2$ for otherwise $M$ is graphic. Let $f$ be an element of $M$ such that $N$ is 3-connected with $U_{2,4}$-minors for some matroid $N\in\{M\del f, M/f\}$. Let $B$ be a basis of $N$ and $G$ a graph with $N=L(G)$. Since $N$ is 3-connected, $G$ is 2-edge-connected and has no degree-2 vertices. Moreover, since $2\leq |E(N)-B|\leq4$ and $G|B$ is a spanning graph with a unique cycle,  the graph $G|B$ has at most four degree-1 vertices and $r(N)=|V(G)|\leq 8$. So $r(M)\leq 9$.
\end{proof}

Elements $x$ and $x^{\prime}$ of a matroid $N$ are {\sl clones} if the function exchanging $x$ with $x^{\prime}$ and fixing every other points in $E(N)$ is an automorphism of $N.$ A set $X\subseteq{E(N)}$ is a {\sl clonal set} of $N$ if every pair of  elements of $X$ are clones.

\begin{lem}\label{long-line}
$M$ has no $U_{2,5}$-restriction.
\end{lem}

\begin{proof}
Assume otherwise. Let $X$ be a subset of $E(M)$ with $M|X=U_{2,n}$ and $n\geq5$. We claim that there is a clonal set $X_1\subseteq X$ of $M$ with $|X_1|\geq3$. Let $e\in X$. Since $M\del e$ is non-graphic, there is a graph $H$ with $M\del e=L(H)$.  Let $X_1$ be the parallel class of $H$ with $X_1\subset X$. Evidently, $|X_1|\geq3$ and $X_1$ is a clonal set of $M\del e$ as $M\del e=L(H)$. Next we prove that $X_1$ is also a clonal set of $M$. Assume otherwise. Then there are $e_1,e_2\in X_1$ and an independent set $I$ of $M$ with $I\subseteq E(M)-X$ such that $I\cup\{e,e_1\}\in\mathcal{C}(M)$ and $I\cup\{e,e_2\}\notin\mathcal{C}(M)$. Since $\{e,e_1,e_2\}\in\mathcal{C}(M)$, there is a circuit $C$ of $M$ with $e_2\in C\subseteq I\cup\{e,e_2\}$. Moreover, since $X_1$ is a clonal set of $M\del e$ and $I\cup e_1$ is independent in $M$, there is a set $I_1\subseteq I$ with $C=I_1\cup\{e,e_2\}$. Since $I\cup\{e,e_2\}\notin\mathcal{C}(M)$, we have $I_1\neq I$. Then there is a circuit $C_1$ of $M$ with $C_1\subseteq (C\cup\{e,e_1,e_2\})-e_2\subseteq I_1\cup\{e,e_1\}$, a contradiction to the fact $I\cup\{e,e_1\}\in\mathcal{C}(M)$. So the claim holds.

Let $G_1$ be a graph with $M\del e_1=L(G_1)$. Since $X_1$ is a clonal set of $M$ with at least three elements, $G_1|X_1-e_1$ is a non-trivial parallel class. Add the edge $e_1$ to $G_1$ get a graph $G$ such that $G|X_1$ is a parallel class. Since $X_1$ is a clonal set of $M$, it is easy to verify that $M=L(G)$, a contradiction.
\end{proof}

\begin{proof}[Proof of Theorem \ref{main-thm}.]
When $M$ is not 3-connected, Theorems \ref{1-con} and \ref{2-con} imply that $M$ is direct sum of $U_{2,4}$ and a loop. So we may assume that $M$ is 3-connected. By Lemmas \ref{rank-corank} and \ref{corank<4} we  have $r(M)\leq 11$. When $r(M^*)\leq 2$, we have $|E(M)|\leq13$. So we may assume that $r(M^*)\geq 3$. By Lemma \ref{cor:del-rank}, there is an element $f\in E(M)$ such that $\co(M\del f)$ is 3-connected with $U_{2,4}$-minors. Let $G$ be a graph with $\co(M\del f)=L(G)$. Since $M$ has no $U_{2,5}$-restriction by Lemma \ref{long-line}, the matroid $\co(M\del f)$ also has no $U_{2,5}$-restriction. Hence, $|E(G)|\leq 4\times \binom{11}{2}=220$ by  $|V(G)|=r(\co(M\del f))\leq r(M)\leq 11$. On the other hand, since $r(\co(M\del f))\geq r(M)-3$ by Lemma \ref{cor:del-rank}, we have $$|E(M)|\leq |E(\co(M\del f))|+4=|E(G)|+4\leq224.$$
\end{proof}



\end{document}